\documentclass[11pt]{article}
\font\smallit=cmti10
\font\smalltt=cmtt10

\usepackage{amsmath}
\usepackage{amsthm}
\usepackage{amssymb}
\usepackage{float}

\usepackage{placeins}
\usepackage{array}
\usepackage{multirow}
\usepackage{hhline}
\usepackage{verbatim}
\usepackage[letterpaper,left=1.5in,right=1in,top=1in,bottom=1in]{geometry}
\usepackage{amsfonts}
\usepackage{enumerate}
\usepackage{graphicx}
\usepackage{tabu}
\usepackage{mathrsfs}
\newcommand\widebar[1]{\mathop{\overline{#1}}}

\usepackage{tikz}
\usepackage{pifont}
\usetikzlibrary{patterns}
\usetikzlibrary{positioning}
\usepackage{ marvosym }
\usepackage{verbatim}
\usepackage[mathscr]{euscript}
\usepackage{tabu}
\usetikzlibrary{arrows}
\usetikzlibrary{shapes,snakes}

\usepackage{setspace,color}
\newtheorem{theorem}{Theorem}
\newtheorem{lemma}[theorem]{Lemma}
\newtheorem{observation}[theorem]{Observation}

\newtheorem{corollary}{Corollary}
\newtheorem{definition}{Definition}

\newcommand{\dollarsign}{L}

\begin{document}
\begin{center}

\uppercase{Complete Graphs and Polyominoes}
\author{Todd Mullen}

\vskip 20pt
{\bf Todd Mullen, Richard Nowakowski, Danielle Cox\footnote{This author was supported by NSERC. }}\\
{\smallit Department of Mathematics and Statistics, Dalhousie University, NS, Canada}\\
{\tt toddmullen26@outlook.com}\\

\end{center}
\vskip 30pt

\centerline{\smallit Received: , Revised: , Accepted: , Published: }
\vskip 30pt

\centerline{\bf Abstract}

In the chip-firing variant, Diffusion, chips flow from places of high concentration to places of low concentration (or equivalently, from the rich to the poor). We explore this model on complete graphs, determining the number of different ways that chips can be distributed on an unlabelled complete graph and demonstrate connections to polyominoes.

\pagestyle{myheadings}
\markright{\smalltt INTEGERS: 16(2016)\hfill}
\thispagestyle{empty}
\baselineskip=12.875pt
\vskip 30pt
\section{Introduction}

Introduced by Duffy et al. \cite{duffy}, Parallel Diffusion is a process defined on a simple finite graph, $G$, in which at every time step, chips are diffused throughout the graph following specific rules. Each vertex is assigned a \textit{stack size} which is an integral number. This number represents the number of chips a vertex has. At each time step, the chips are redistributed via the following rules: If a vertex is adjacent to a vertex with fewer chips, it takes a chip from its stack and sends it to the poorer vertex. 
We call this action the \textit{firing} of a vertex. At each time step of the diffusion process, every vertex fires. Note that when a vertex with no poorer neighbours fires, it does not send any chips. In Figure~\ref{fig:parex}, we see an example of Parallel Diffusion. Note that at each time step, the vertices of $P_5$ have a stack size; an assignment of stack sizes to the vertices of a graph $G$ is referred to as a \textit{configuration}, denoted $C=\{(v,|v|^{C}): v\in V(G)\}$, where $|v|^C$ is the stack size of $v$ in $C$. We omit the superscript when the configuration is clear.

A vertex $v$ is said to be \textit{richer} than another vertex $u$ in configuration $C$ if $|v|^C > |u|^C$. In this instance, $u$ is said to be \textit{poorer} than $v$ in $C$. If $|v|^C<0$, we say $v$ is \textit{in debt in $C$}.

In Parallel Diffusion, given a graph $G$ and a configuration $C$ on $G$, to \textit{fire} $C$ is to decrease the stack size of every vertex $v \in V(G)$ by the number of poorer neighbours $v$ has and increase the stack size of $v$ by the number of richer neighbours $v$ has. More formally, for all $v$, let $Z_{-}^{C}(v) = \{u \in N(v) : |v|^{C} > |u|^{C}\}$ and let $Z_{+}^{C}(v) = \{u \in N(v) : |u|^{C} > |v|^{C}\}$. Then, firing results in every vertex $v$ changing from a stack size of $|v|^{C}$ to a stack size of $|v|^{C} + |Z_{+}^{C}(v)| - |Z_{-}^{C}(v)|$. 

We are interested in counting the number of configurations on $K_n$, $n \geq 1$ (up to a definition of equivalence). We will show, in Theorem~\ref{thm:polyominoisomorphism}, a bijection between the number of configurations that exist on unlabelled complete graphs of order $n$ and the number of board-pile $n$-ominoes (or sets of stacked $1 \times k$ rectangles with a total area of $n$). An example of a board-pile $n$-omino is given in Figure~\ref{fig:polyominoexample1}, and an example of a configuration on $K_n$ is given in Figure~\ref{fig:completeconfiguration1}.
i
\begin{figure}
	\[
	\begin{tikzpicture}[-,-=stealth', auto,node distance=1.5cm,
	thick,scale=0.6, main node/.style={scale=0.6,circle,draw}]
	
	\node[main node, label={[red]90:0}] (1) 					 {$v_5$};					 
	\node[main node, label={[red]90:2}] (2)  [right of=1]        {$v_4$};
	\node[main node, label={[red]90:0}] (3)  [right of=2]        {$v_3$};  
	\node[main node, label={[red]90:4}] (4) 	[right of=3]     {$v_2$};					 
	\node[main node, label={[red]90:1}] (5)  [right of=4]        {$v_1$}; 
	\node[draw=none, fill=none] (100)  [right of=5]              {$C_0$};
	
	\path[every node/.style={font=\sffamily\small}]		 
	
	(1) edge node [] {} (2)			
	(2) edge node [] {} (3)
	(3) edge node [] {} (4)
	(4) edge node [] {} (5);

	\end{tikzpicture} \]
	Initial firing
	\FloatBarrier

	\[
	\begin{tikzpicture}[-,-=stealth', auto,node distance=1.5cm,
	thick,scale=0.6, main node/.style={scale=0.6,circle,draw}]
	
	\node[main node, label={[red]90:1}] (1) 					    {$v_5$};					 
	\node[main node, label={[red]90:0}] (2)  [right of=1]        {$v_4$};
	\node[main node, label={[red]90:2}] (3)  [right of=2]        {$v_3$};  
	\node[main node, label={[red]90:2}] (4) 	[right of=3]        {$v_2$};					 
	\node[main node, label={[red]90:2}] (5)  [right of=4]        {$v_1$}; 
	\node[draw=none, fill=none] (100)  [right of=5]              {$C_1$};
	
	\path[every node/.style={font=\sffamily\small}]		 
	
	(1) edge node [] {} (2)			
	(2) edge node [] {} (3)
	(3) edge node [] {} (4)
	(4) edge node [] {} (5);

	\end{tikzpicture} \]
	Firing at step 1
	\FloatBarrier

	\[
	\begin{tikzpicture}[-,-=stealth', auto,node distance=1.5cm,
	thick,scale=0.6, main node/.style={scale=0.6,circle,draw}]
	
	\node[main node, label={[red]90:0}] (1) 					    {$v_5$};					 
	\node[main node, label={[red]90:2}] (2)  [right of=1]        {$v_4$};
	\node[main node, label={[red]90:1}] (3)  [right of=2]        {$v_3$};  
	\node[main node, label={[red]90:2}] (4) 	[right of=3]        {$v_2$};					 
	\node[main node, label={[red]90:2}] (5)  [right of=4]        {$v_1$}; 
	\node[draw=none, fill=none] (100)  [right of=5]              {$C_2$};
	
	\path[every node/.style={font=\sffamily\small}]		 
	
	(1) edge node [] {} (2)			
	(2) edge node [] {} (3)
	(3) edge node [] {} (4)
	(4) edge node [] {} (5);

	\end{tikzpicture} \]
	Firing at step 2
	\FloatBarrier

	\[
	\begin{tikzpicture}[-,-=stealth', auto,node distance=1.5cm,
	thick,scale=0.6, main node/.style={scale=0.6,circle,draw}]
	
	\node[main node, label={[red]90:1}] (1) 					    {$v_5$};					 
	\node[main node, label={[red]90:0}] (2)  [right of=1]        {$v_4$};
	\node[main node, label={[red]90:3}] (3)  [right of=2]        {$v_3$};  
	\node[main node, label={[red]90:1}] (4) 	[right of=3]        {$v_2$};					 
	\node[main node, label={[red]90:2}] (5)  [right of=4]        {$v_1$}; 
	\node[draw=none, fill=none] (100)  [right of=5]              {$C_3$};	
	
	\path[every node/.style={font=\sffamily\small}]		 
	
	(1) edge node [] {} (2)			
	(2) edge node [] {} (3)
	(3) edge node [] {} (4)
	(4) edge node [] {} (5);

	\end{tikzpicture} \]
	Firing at step 3
	\FloatBarrier

	\[
	\begin{tikzpicture}[-,-=stealth', auto,node distance=1.5cm,
	thick,scale=0.6, main node/.style={scale=0.6,circle,draw}]
	
	\node[main node, label={[red]90:0}] (1) 					    {$v_5$};					 
	\node[main node, label={[red]90:2}] (2)  [right of=1]        {$v_4$};
	\node[main node, label={[red]90:1}] (3)  [right of=2]        {$v_3$};  
	\node[main node, label={[red]90:3}] (4) 	[right of=3]        {$v_2$};					 
	\node[main node, label={[red]90:1}] (5)  [right of=4]        {$v_1$}; 
	\node[draw=none, fill=none] (100)  [right of=5]              {$C_4$};	
	
	\path[every node/.style={font=\sffamily\small}]		 
	
	(1) edge node [] {} (2)			
	(2) edge node [] {} (3)
	(3) edge node [] {} (4)
	(4) edge node [] {} (5);

	\end{tikzpicture} \]
	Firing at step 4
	\FloatBarrier

	\[
	\begin{tikzpicture}[-,-=stealth', auto,node distance=1.5cm,
	thick,scale=0.6, main node/.style={scale=0.6,circle,draw}]
	
	\node[main node, label={[red]90:1}] (1) 					    {$v_5$};					 
	\node[main node, label={[red]90:0}] (2)  [right of=1]        {$v_4$};
	\node[main node, label={[red]90:3}] (3)  [right of=2]        {$v_3$};  
	\node[main node, label={[red]90:1}] (4) 	[right of=3]        {$v_2$};					 
	\node[main node, label={[red]90:2}] (5)  [right of=4]        {$v_1$}; 
	\node[draw=none, fill=none] (100)  [right of=5]              {$C_5$};	
	
	\path[every node/.style={font=\sffamily\small}]		 
	
	(1) edge node [] {} (2)			
	(2) edge node [] {} (3)
	(3) edge node [] {} (4)
	(4) edge node [] {} (5);

	\end{tikzpicture} \]
	Firing at step 5
	\FloatBarrier

	\[
	\begin{tikzpicture}[-,-=stealth', auto,node distance=1.5cm,
	thick,scale=0.6, main node/.style={scale=0.6,circle,draw}]
	
	\node[main node, label={[red]90:0}] (1) 					    {$v_5$};					 
	\node[main node, label={[red]90:2}] (2)  [right of=1]        {$v_4$};
	\node[main node, label={[red]90:1}] (3)  [right of=2]        {$v_3$};  
	\node[main node, label={[red]90:3}] (4) 	[right of=3]        {$v_2$};					 
	\node[main node, label={[red]90:1}] (5)  [right of=4]        {$v_1$}; 
	\node[draw=none, fill=none] (100)  [right of=5]              {$C_6$};
	\path[every node/.style={font=\sffamily\small}]		 
	
	(1) edge node [] {} (2)			
	(2) edge node [] {} (3)
	(3) edge node [] {} (4)
	(4) edge node [] {} (5);

	\end{tikzpicture} \]
	
	\FloatBarrier
	\caption{Several steps in a Parallel Diffusion game on $P_5$. The period begins with $C_3$. This is the first configuration that is repeated.}
	\label{fig:parex}
\end{figure}
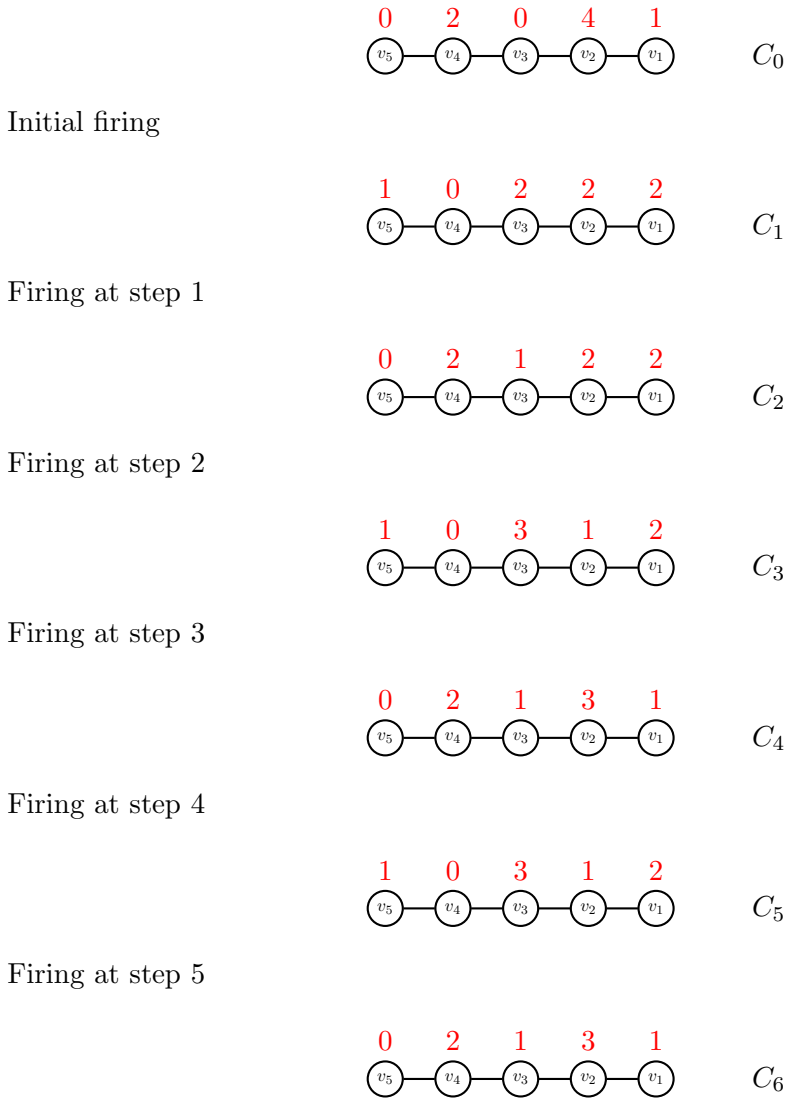
\FloatBarrier

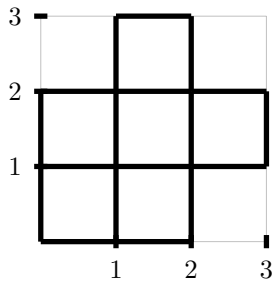
\begin{figure}[H]
	\centering	
	\begin{tikzpicture}[line width = 2pt,scale = 1.0]
	\path [draw, help lines, opacity=.5]  (0,0) grid (3,3);
	\foreach \i in {1,...,3} \draw (\i,2.5pt) -- +(0,-5pt) node [anchor=north, font=\small] {$\i$} (2.5pt,\i) -- +(-5pt,0) node [anchor=east, font=\small] {$\i$};
	\draw (0,0) -- (1,0);
	\draw (1,0) -- (1,1);
	\draw (1,1) -- (2,1);
	\draw (0,0) -- (0,1);
	\draw (0,1) -- (1,1);
	\draw (1,0) -- (2,0);
	\draw (2,0) -- (2,1);
	\draw (2,1) -- (3,1);
	\draw (2,2) -- (2,1);
	\draw (1,2) -- (2,2);
	\draw (3,1) -- (3,2);
	\draw (2,2) -- (3,2);
	\draw (0,1) -- (0,2);
	\draw (2,2) -- (2,3);
	\draw (1,3) -- (2,3);
	\draw (1,3) -- (1,2);
	\draw (1,1) -- (1,2);
	\draw (0,2) -- (1,2);

	\end{tikzpicture}
	
	\caption{Board-pile 6-omino $X$}
	\label{fig:polyominoexample1}
\end{figure}

\FloatBarrier

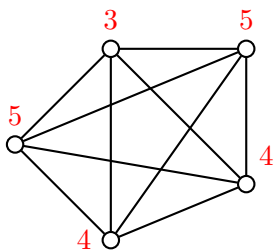
\begin{figure}[H]
	\centering
	\begin{tikzpicture}[-,-=stealth', auto,node distance=3cm,
	thick,scale=0.6, main node/.style={scale=0.6,circle,draw}]                               
	\node[main node, label={[red]90:3}] (1)          	{};						 
	\node[main node, label={[red]90:5}] (2)  [right of=1]        {};
	\node[main node, label={[red]90:5}] (3)  [below left of=1]        {};  
	\node[main node, label={[red]75:4}] (4)  [below of=2]        {};
	\node[main node, label={[red]180:4}] (5)  [below right of=3]        {}; 	
	
	\path
	(1) edge (2)
	(1) edge (3)
	(1) edge (4)
	(2) edge (3)
	(2) edge (4)
	(2) edge (5)
	(3) edge (4)
	(3) edge (5)
	(4) edge (5)
	(5) edge (1);

	\end{tikzpicture}

	\caption{Configuration on a complete graph, $K_5$.}
	\label{fig:completeconfiguration1}
\end{figure}

Unlike some other chip-firing processes like the original Chip-Firing game \cite{lovasz} and Brushing \cite{pralat}, in Parallel Diffusion it is possible for a stack size to initially be positive but to become negative as time goes on. For example if some vertex $v$ with a stack size of $n$, $n \in \mathbb{N}$, is adjacent to $n+1$ vertices, each of which having a stack size of $0$, then after firing, $v$ would have at stack size of $-1$. However in \cite{duffy}, it was shown that Parallel Diffusion is such that an addition of some constant $k$, $k \in \mathbb{Z}$, to each stack size will have no effect on determining when and if a chip will move from one vertex to another. So if one wanted to view diffusion as a process in which stack sizes are never negative, one would only need to add a sufficient constant $k$, $k \in \mathbb{N}$, to each stack size. Some results pertaining to locating an appropriate $k$ value for any given graph can be found in \cite{carlotti}.

\section{Diffusion Background}

We begin with some necessary terminology.

Let $G$ be a finite simple undirected graph with vertex set $V(G)$ and edge set $E(G)$. Let $A \subseteq E(G)$. A \textit{graph orientation} of a graph $G$ is a mixed graph obtained from $G$ by choosing an orientation ($x \to y$ or $y \to x$) for each edge $xy \in A$. We refer to the edges that are in $E(G)$\textbackslash$A$ as \textit{flat}. We refer to the assignment of either $x \to y$, $y \to x$, or flat to an edge $xy$ as $xy$'s \textit{edge orientation}.

We refer to the discrete time increments in Parallel Diffusion as \textit{steps}. The configuration of $G$ at any step $k$ is referred to as the \textit{configuration at step $k$} or the \textit{configuration at $t = k$}. We refer to the configuration at step $0$ as the initial configuration. At every step, the vertices of $V(G)$ fire. So, a step $t$ consists of both a configuration and the firing of the vertices yielding the configuration at step $t+1$. The firing of vertices at step 0 which yields the configuration at step 1 is called the \textit{initial firing}.

In Parallel Diffusion, the assigned value of a vertex, $v$, at step $t$, is referred to as its \textit{stack size at time $t$}. If the initial configuration is $C$, then the stack size at time $t$ is denoted $|v|_t^C$. This implies that $|v|^C = |v|_0^C$. We omit the superscript when the configuration is clear.

In Parallel Diffusion, given a graph $G$ and an initial configuration $C_0$, the configuration at time $t$ can be denoted by $C_t=\{(v,|v|^{C}_{t}): v\in V(G)\}$. The \textit{configuration sequence} $Seq(C_0) = (C_0,C_1,C_2, \dots)$ is the sequence of configurations that arises as the time increases. The configuration sequence clearly depends on both the initial configuration and the graph $G$. However, $G$ is omitted from the notation since it will always be clear to which graph we are referring. 

Let $Seq(C_0) = (C_0,C_1,C_2, \dots)$ be the configuration sequence on a graph $G$ with initial configuration $C_0$. A positive integer $p$ is a \textit{period length} if $C_t = C_{t+p}$ for all $t \geq N$ for some $N$. In this case, $N$ is a \textit{preperiod length}. For such a value, $N$, if $k \geq N$, then we say that the configuration, $C_k$, is \textit{inside} the period. For the purposes of this paper, all references to period length will refer to the \textit{minimum period length $p$} in a given configuration sequence. Also, all references to preperiod length will refer to the \textit{least preperiod length} that yields that minimum period length $p$ in a given configuration sequence. Given two configurations, $C$ and $D$, of a graph $G$, in which the vertices are labelled, $C$ and $D$ are \textit{equal} if $|v|^C = |v|^D$ for all $v \in V(G)$. In Figure~\ref{fig:parex}, the period length is 2 and the preperiod length is 3.

Let \textit{\dollarsign(G)} = \{$Seq(C): C$ is a configuration on $G$\}. Since the set of integers is infinite, on any graph $G$, \dollarsign(G) is an infinite set. 

In their paper \cite{long}, Long and Narayanan prove that the period length of every configuration sequence is either 1 or 2 \cite{long}. Let $\widebar{Seq(C_0)}$ denote the singleton or ordered pair of configurations contained within the period of a configuration sequence $Seq(C_0)$.

Let $C$ be a configuration on a graph $G$. Let $C+k$ denote the configuration created by adding an integer $k$ to every stack size in the configuration $C$. Two configuration sequences, $Seq(C)$ and $Seq(D)$ in $\dollarsign(G)$, are \textit{equivalent} if $\widebar{Seq}(C+k) = \widebar{Seq}(D)$ for some integer $k$. For all configurations $C$ and all integers $k$, we say that $C$ and $C+k$ are \textit{equivalent}.

We see an example of equivalent configuration sequences in Figure~\ref{fig:isomorphic}.

A configuration $D$ on a graph $G$ is a \textit{period configuration} if $D$ is in $\widebar{Seq}(C)$ for some configuration $C$. A configuration $D$ on a graph $G$ is a \textit{$p_2$-configuration} if $D$ is in  $\widebar{Seq}(C)$ for some configuration $C$ and $\widebar{Seq}(C)$ has exactly 2 elements. A configuration $D$ on a graph $G$ is a \textit{fixed configuration} if $D$ is in $\widebar{Seq}(C)$ for some configuration $C$ and $\widebar{Seq}(C)$ has exactly 1 element.

\begin{figure}[H]
	
	Configuration sequence $Seq(C_0)$
	
	
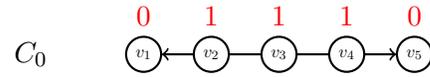
\[
	\begin{tikzpicture}[-,-=stealth', auto,node distance=1.5cm,
	thick,scale=0.6, main node/.style={scale=0.6,circle,draw}]
	
	\node[main node, label={[red]90:0}] (1) 					   {$v_1$};						 
	\node[main node, label={[red]90:1}] (2)  [right of=1]        {$v_2$};
	\node[main node, label={[red]90:1}] (3)  [right of=2]        {$v_3$};  
	\node[main node, label={[red]90:1}] (4)  [right of=3]        {$v_4$};						 
	\node[main node, label={[red]90:0}] (5)  [right of=4]        {$v_5$}; 
	\node[draw=none, fill=none]         (0)  [left of=1]         {$C_0$};
	\path
	
	(2) edge (3)
	(3) edge (4); 
	
	\draw [->] (2) edge (1) 
	(4) edge (5) ;

	\end{tikzpicture} \]

	\FloatBarrier

	Firing at step 0

	
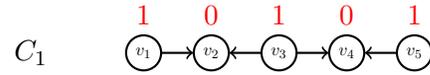
\[
	\begin{tikzpicture}[-,-=stealth', auto,node distance=1.5cm,
	thick,scale=0.6, main node/.style={scale=0.6,circle,draw}]
	
	\node[main node, label={[red]90:1}] (1) 					   {$v_1$};						 
	\node[main node, label={[red]90:0}] (2)  [right of=1]        {$v_2$};
	\node[main node, label={[red]90:1}] (3)  [right of=2]        {$v_3$};  
	\node[main node, label={[red]90:0}] (4)  [right of=3]        {$v_4$};						 
	\node[main node, label={[red]90:1}] (5)  [right of=4]        {$v_5$}; 
	\node[draw=none, fill=none]         (0)  [left of=1]         {$C_1$};
	\draw [->] (1) edge (2) (3) edge (2) (3) edge (4)
	(5) edge (4) ;

	\end{tikzpicture} \]
	
	\FloatBarrier

	Firing at step 1
	
	
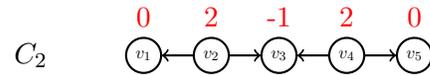
\[
	\begin{tikzpicture}[-,-=stealth', auto,node distance=1.5cm,
	thick,scale=0.6, main node/.style={scale=0.6,circle,draw}]
	
	\node[main node, label={[red]90:0}] (1) 					   {$v_1$};						 
	\node[main node, label={[red]90:2}] (2)  [right of=1]        {$v_2$};
	\node[main node, label={[red]90:-1}] (3)  [right of=2]        {$v_3$};  
	\node[main node, label={[red]90:2}] (4)  [right of=3]        {$v_4$};						 
	\node[main node, label={[red]90:0}] (5)  [right of=4]        {$v_5$}; 
	\node[draw=none, fill=none]         (0)  [left of=1]         {$C_2$};
	\draw [->] (2) edge (1) (2) edge (3) (4) edge (3)
	(4) edge (5) ;

	\end{tikzpicture} \]
	
	\FloatBarrier
	
	Firing at step 2
	
	
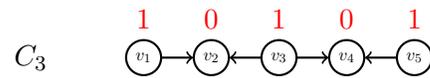
\[
	\begin{tikzpicture}[-,-=stealth', auto,node distance=1.5cm,
	thick,scale=0.6, main node/.style={scale=0.6,circle,draw}]
	
	\node[main node, label={[red]90:1}] (1) 					   {$v_1$};						 
	\node[main node, label={[red]90:0}] (2)  [right of=1]        {$v_2$};
	\node[main node, label={[red]90:1}] (3)  [right of=2]        {$v_3$};  
	\node[main node, label={[red]90:0}] (4)  [right of=3]        {$v_4$};						 
	\node[main node, label={[red]90:1}] (5)  [right of=4]        {$v_5$}; 
	\node[draw=none, fill=none]         (0)  [left of=1]         {$C_3$};
	\draw [->] (1) edge (2) (3) edge (2) (3) edge (4)
	(5) edge (4) ;

	\end{tikzpicture} \]
	
	\FloatBarrier

	\vspace{1cm}

	Configuration sequence $Seq(C'_0)$
	
	
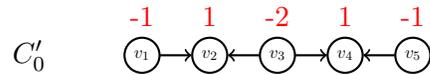
\[
	\begin{tikzpicture}[-,-=stealth', auto,node distance=1.5cm,
	thick,scale=0.6, main node/.style={scale=0.6,circle,draw}]
	
	\node[main node, label={[red]90:-1}] (1) 					   {$v_1$};						 
	\node[main node, label={[red]90:1}] (2)  [right of=1]        {$v_2$};
	\node[main node, label={[red]90:-2}] (3)  [right of=2]        {$v_3$};  
	\node[main node, label={[red]90:1}] (4)  [right of=3]        {$v_4$};						 
	\node[main node, label={[red]90:-1}] (5)  [right of=4]        {$v_5$}; 
	\node[draw=none, fill=none]         (0)  [left of=1]         {$C'_0$};
	\draw [->] (1) edge (2) (3) edge (2) (3) edge (4)
	(5) edge (4) ;

	\end{tikzpicture} \]
	
	\FloatBarrier
	
	Firing at step 0
	
	
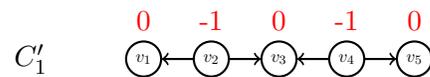
\[
	\begin{tikzpicture}[-,-=stealth', auto,node distance=1.5cm,
	thick,scale=0.6, main node/.style={scale=0.6,circle,draw}]
	
	\node[main node, label={[red]90:0}] (1) 					   {$v_1$};						 
	\node[main node, label={[red]90:-1}] (2)  [right of=1]        {$v_2$};
	\node[main node, label={[red]90:0}] (3)  [right of=2]        {$v_3$};  
	\node[main node, label={[red]90:-1}] (4)  [right of=3]        {$v_4$};						 
	\node[main node, label={[red]90:0}] (5)  [right of=4]        {$v_5$}; 
	\node[draw=none, fill=none]         (0)  [left of=1]         {$C'_1$};
	
	\draw [->] (2) edge (1) (2) edge (3) (4) edge (3) 
	(4) edge (5) ;

	\end{tikzpicture} \]
	
	\FloatBarrier
	
	Firing at step 1
	
	\[
	\begin{tikzpicture}[-,-=stealth', auto,node distance=1.5cm,
	thick,scale=0.6, main node/.style={scale=0.6,circle,draw}]
	
	\node[main node, label={[red]90:-1}] (1) 					   {$v_1$};						 
	\node[main node, label={[red]90:1}] (2)  [right of=1]        {$v_2$};
	\node[main node, label={[red]90:-2}] (3)  [right of=2]        {$v_3$};  
	\node[main node, label={[red]90:1}] (4)  [right of=3]        {$v_4$};						 
	\node[main node, label={[red]90:-1}] (5)  [right of=4]        {$v_5$}; 
	\node[draw=none, fill=none]         (0)  [left of=1]         {$C'_2$};
	\draw [->] (1) edge (2) (3) edge (2) (3) edge (4)
	(5) edge (4);

	\end{tikzpicture} \]
	
	\FloatBarrier
	
	Firing at step 2
	
	\[
	\begin{tikzpicture}[-,-=stealth', auto,node distance=1.5cm,
	thick,scale=0.6, main node/.style={scale=0.6,circle,draw}]
	
	\node[main node, label={[red]90:0}] (1) 					   {$v_1$};						 
	\node[main node, label={[red]90:-1}] (2)  [right of=1]        {$v_2$};
	\node[main node, label={[red]90:0}] (3)  [right of=2]        {$v_3$};  
	\node[main node, label={[red]90:-1}] (4)  [right of=3]        {$v_4$};						 
	\node[main node, label={[red]90:0}] (5)  [right of=4]        {$v_5$}; 
	\node[draw=none, fill=none]         (0)  [left of=1]         {$C'_3$};
	\draw [->] (2) edge (1) (2) edge (3) (4) edge (3)
	(4) edge (5) ;

	\end{tikzpicture} \]
	
	\FloatBarrier
	
	\caption{Two equivalent configuration sequences.}
	
	\label{fig:isomorphic}
	
\end{figure}

We conclude this section with an observation tying together Parallel Diffusion and graph orientations.

\begin{observation}\label{lem:inducego}
	In Parallel Diffusion, every configuration induces a graph orientation.
\end{observation}

\begin{proof}
	Let $G$ be a graph and $C_t$ a configuration on $G$. For all pairs of adjacent vertices $u$, $v$ in $G$ at step $t$, assign directions as follows:
	
	\begin{itemize}
		
		\item If $|u|_t > |v|_t$, assign $uv$ the edge orientation $u \to v$.
		
		\item If $|v|_t > |u|_t$, assign $uv$ the edge orientation $v \to u$.
		
		\item If $|u|_t = |v|_t$, do not direct the edge $uv$.
		
	\end{itemize}
	
\end{proof}

We say that this graph orientation is \textit{induced} by $C_t$. We see an example of a graph orientation induced by a configuration in Parallel Diffusion in Figure~\ref{fig:underlyingorientation}.

\begin{figure} [H] 
	\centering	
	
	\begin{tikzpicture}[-,-=stealth', auto,node distance=1.5cm,
	thick,scale=0.6, main node/.style={scale=0.6,circle,draw}]
	
	\node[main node, label={[red]90:15}] (1) 					   {$v_1$};						 
	\node[main node, label={[red]90:9}] (2)  [right of=1]        {$v_2$};
	\node[main node, label={[red]90:8}] (3)  [right of=2]        {$v_3$};  
	\node[main node, label={[red]90:2}] (4)  [right of=3]        {$v_4$};						 
	\node[main node, label={[red]90:12}] (5)  [right of=4]        {$v_5$}; 
	\node[main node, label={[red]90:}] (6) 	[below of=1]				   {$v_1$};						 
	\node[main node, label={[red]90:}] (7)  [right of=6]        {$v_2$};
	\node[main node, label={[red]90:}] (8)  [right of=7]        {$v_3$};  
	\node[main node, label={[red]90:}] (9)  [right of=8]        {$v_4$};						 
	\node[main node, label={[red]90:}] (10)  [right of=9]        {$v_5$}; 	
	
	\path (1) edge (2) (2) edge (3) (3) edge (4) (4) edge (5);
	
	\draw [->] (6) edge (7) (7) edge (8) (8) edge (9)
	(10) edge (9) ;

	\end{tikzpicture} 
	\FloatBarrier
	
	\caption{Configuration on $P_5$ and its induced graph orientation.}
	\label{fig:underlyingorientation}

\end{figure}
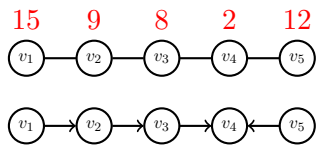

\section{Polyominoes}

We will now introduce polyominoes. The following concepts are from David Klarner's paper \cite{klarner}, reworded slightly to improve the clarity of our results.

\begin{definition} \cite{klarner}
	A \textbf{polyomino} is a plane figure composed of a number of connected unit squares joined edge on edge. A polyomino with exactly $n$ unit squares is called an \textbf{n-omino}.
\end{definition}

\begin{definition} \cite{klarner}
	In a polyomino $X$, a \textbf{horizontal strip}, or \textbf{h-strip}, is a maximal rectangle of height one. 
\end{definition}

By convention, we will set each $h$-strip in the plane so that its height spans from an integer $k$ to $k+1$.

\begin{definition} \cite{klarner}
	The infinite area enclosed by the lines $y=k$ and $y=k+1$ is called a \textbf{row}.	
\end{definition}	

\begin{definition} \cite{klarner}
	A \textbf{board-pile polyomino} is a polyomino which has at most one $h$-strip per row. A board-pile polyomino with $n$ unit squares is called a \textbf{board-pile n-omino}.
\end{definition}

With this, we can now begin to prove that there exists a bijection between the number of board-pile $n$-ominoes and the number of period configurations of an unlabelled complete graph on $n$ vertices. To accomplish this, we will first develop a notation for polyominoes that will eliminate the necessity of a pictorial representation, then define a mapping from the set of all polyominoes on $n$ unit squares to the set of all period configurations of (an unlabelled) $K_n$ up to equivalence, and then show that mapping to be a bijection.

Given a polyomino $X$, we will use the convention of labelling the h-strips from bottom to top as $S_1$, $S_2$, $\dots$, $S_N$, where $N$ is the number of h-strips in $X$. Let $S(X)$ be the set of all h-strips in $X$.

A board-pile polyomino $X$, can be represented as an ordered set of ordered pairs of the form 
$X = \{ (d_1, |S_1|), (d_2, |S_2|), (d_3, |S_3|), \dots (d_N, |S_{N}|)\}$, where $|S_i|$ is the number of unit squares in the h-strip $S_i$, and $d_i$ is the difference between the greatest $x$-coordinate in $S_i$ and the least $x$-coordinate in $S_{i-1}$. By convention, $d_1 = 0$. See Figure~\ref{fig:polyominoexample} for an example.

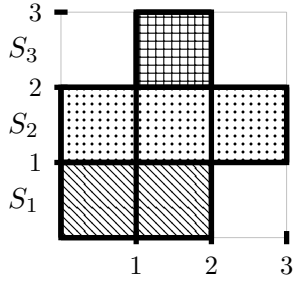
\begin{figure}[H]
	\centering	
	\begin{tikzpicture}[line width = 2pt,scale = 1.0]
	\path [draw, help lines, opacity=.5]  (0,0) grid (3,3);
	\foreach \i in {1,...,3} \draw (\i,2.5pt) -- +(0,-5pt) node [anchor=north, font=\small] {$\i$} (2.5pt,\i) -- +(-5pt,0) node [anchor=east, font=\small] {$\i$};
	\draw (0,0) -- (1,0);
	\draw (1,0) -- (1,1);
	\draw (1,1) -- (2,1);
	\draw (0,0) -- (0,1);
	\draw (0,1) -- (1,1);
	\draw (1,0) -- (2,0);
	\draw (2,0) -- (2,1);
	\draw (2,1) -- (3,1);
	\draw (2,2) -- (2,1);
	\draw (1,2) -- (2,2);
	\draw (3,1) -- (3,2);
	\draw (2,2) -- (3,2);
	\draw (0,1) -- (0,2);
	\draw (2,2) -- (2,3);
	\draw (1,3) -- (2,3);
	\draw (1,3) -- (1,2);
	\draw (1,1) -- (1,2);
	\draw (0,2) -- (1,2);
	
	\draw[pattern=north west lines] (0,0) rectangle (2,1);
	
	\draw[pattern=dots] (0,1) rectangle (3,2);	
	
	\draw[pattern=grid] (1,2) rectangle (2,3);

	\node[draw=none] at (-0.5,0.5) {$S_1$};
	\node[draw=none] at (-0.5,1.5) {$S_2$};
	\node[draw=none] at (-0.5,2.5) {$S_3$};
	\end{tikzpicture}

	$$
	X = \{ (0, 2), (3, 3), (2, 1)\}
	$$
	\\
	$$
	S(X) = \{S_1, S_2, S_3\}
	$$
	
	\caption{Board-pile 6-omino $X$ with shading differentiating between $S_1$, $S_2$, and $S_3$.}
	\label{fig:polyominoexample}
\end{figure}

\FloatBarrier

\section{Period Configurations on Unlabelled Complete Graphs}\label{sec:p_2config}

In this section, we will be counting the number of period configurations that exist on unlabelled complete graphs.

A configuration, $C$, of an unlabelled $K_n$, $n \geq 1$, can be represented by a multi-set of cardinality $n$ with the stack sizes as elements. We will use the notation $a^k$ to represent $k$ instances of stack size $a$ in the configuration. An example is shown in Figure~\ref{fig:completeconfiguration}. 

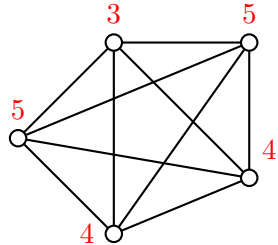
\begin{figure}[H]
	\centering
	\begin{tikzpicture}[-,-=stealth', auto,node distance=3cm,
	thick,scale=0.6, main node/.style={scale=0.6,circle,draw}]                               
	\node[main node, label={[red]90:3}] (1)          	{};						 
	\node[main node, label={[red]90:5}] (2)  [right of=1]        {};
	\node[main node, label={[red]90:5}] (3)  [below left of=1]        {};  
	\node[main node, label={[red]75:4}] (4)  [below of=2]        {};
	\node[main node, label={[red]180:4}] (5)  [below right of=3]        {}; 	
	
	\path
	(1) edge (2)
	(1) edge (3)
	(1) edge (4)
	(2) edge (3)
	(2) edge (4)
	(2) edge (5)
	(3) edge (4)
	(3) edge (5)
	(4) edge (5)
	(5) edge (1);

	\end{tikzpicture}
	
	\begin{align*}
		C &= \{3,4,4,5,5\}\\
		  &= \{3,4^2,5^2\}
	\end{align*}

	\caption{Configuration on an unlabelled complete graph}
	\label{fig:completeconfiguration}
\end{figure}

Given $n$, define a map $f$ from the set of all board-pile $n$-ominoes to the set of all complete graph configurations on $K_n$. We will show that only period configurations of $K_n$ exist in the range of $f$ and that $f$ is a bijection. 

For a board-pile $n$-omino $X = \{ (d_1, |S_1|), (d_2, |S_2|), (d_3, |S_3|), \dots (d_N, |S_{N}|)\}$, let 

\begin{align*}
f(X) &=
 f(\{ (d_1, |S_1|), (d_2, |S_2|), (d_3, |S_3|), \dots (d_n, |S_{N}|)\})\\ &= \Big\{0^{|S_1|},\Big( \sum_{i=1}^{2} d_{i} \Big) ^{|S_2|}, \Big( \sum_{i=1}^{3} d_{i} \Big) ^{|S_3|}, \dots \Big( \sum_{i=1}^{N} d_{i} \Big) ^{|S_{N}|} \Big\}.
\end{align*}

For all $S_k \in S(X)$, the complete graph configuration $f(X)$ has $|S_k|$ vertices, each of which contain $\sum_{i=1}^{k} d_{i}$ chips.

We denote the set of vertices in $f(X)$ corresponding to the strip $S_k$ to be $V_k$ for all $k \leq N$. An example of this mapping is shown in Figure~\ref{fig:completemappingexample}.

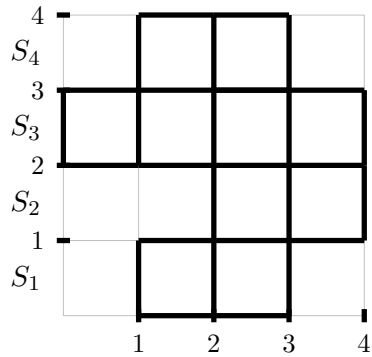
\begin{figure}[H]
	\centering	
	\begin{tikzpicture}[line width = 2pt,scale = 1.0]
	\path [draw, help lines, opacity=.5]  (0,0) grid (4,4);
	\foreach \i in {1,...,4} \draw (\i,2.5pt) -- +(0,-5pt) node [anchor=north, font=\small] {$\i$} (2.5pt,\i) -- +(-5pt,0) node [anchor=east, font=\small] {$\i$};
	\draw (1,0) -- (2,0);
	\draw (2,0) -- (2,1);
	\draw (2,1) -- (3,1);
	\draw (1,0) -- (1,1);
	\draw (1,1) -- (2,1);
	\draw (2,0) -- (3,0);
	\draw (3,0) -- (3,1);
	\draw (3,1) -- (4,1);
	\draw (3,2) -- (3,1);
	\draw (2,2) -- (3,2);
	\draw (4,1) -- (4,2);
	\draw (3,2) -- (4,2);
	\draw (3,2) -- (3,3);
	\draw (2,3) -- (3,3);
	\draw (2,3) -- (2,2);
	\draw (2,1) -- (2,2);
	\draw (1,2) -- (2,2);
	\draw (0,2) -- (0,3);
	\draw (0,2) -- (1,2);
	\draw (0,3) -- (1,3);
	\draw (1,2) -- (1,3);
	\draw (1,3) -- (2,3);
	\draw (1,3) -- (1,4);
	\draw (1,4) -- (2,4);
	\draw (2,3) -- (2,4);
	\draw (2,4) -- (3,4);
	\draw (3,4) -- (3,3);
	\draw (3,3) -- (2,3);
	\draw (3,3) -- (4,3);
	\draw (4,2) -- (4,3);											
	\node[draw=none] at (-0.5,0.5) {$S_1$};
	\node[draw=none] at (-0.5,1.5) {$S_2$};
	\node[draw=none] at (-0.5,2.5) {$S_3$};
	\node[draw=none] at (-0.5,3.5) {$S_4$};
	\end{tikzpicture}

	$$
	X = \{ (0, 2), (3, 2), (2, 4), (3, 2)\}
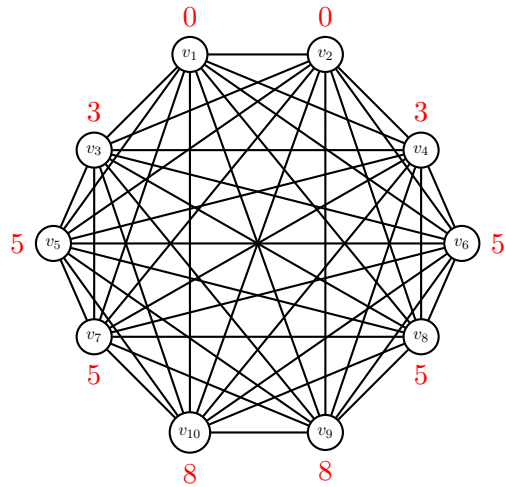
	$$

	\begin{tikzpicture}[-,-=stealth', auto,node distance=3cm,
	thick,scale=0.6, main node/.style={scale=0.6,circle,draw}]                               
	\node[main node, label={[red]90:0}] (1)          	{$v_1$};						 
	\node[main node, label={[red]90:0}] (2)  [right of=1]        {$v_2$};
	\node[main node, label={[red]90:3}] (3)  [below left of=1]        {$v_3$};  
	\node[main node, label={[red]90:3}] (4)  [below right of=2]        {$v_4$};
	\node[main node, label={[red]180:5}] (5)  [below left =0.9cm and 0.2cm of 3]        {$v_5$}; 	
	\node[main node, label={[red]0:5}] (6)  [below right =0.9cm and 0.2cm of 4]        	{$v_6$};						 
	\node[main node, label={[red]270:5}] (7)  [below right =0.9cm and 0.2cm of 5]        {$v_7$};
	\node[main node, label={[red]270:5}] (8)  [below left =0.9cm and 0.2cm of 6]        {$v_8$};  
	\node[main node, label={[red]270:8}] (9)  [below left of=8]        {$v_9$};
	\node[main node, label={[red]270:8}] (10)  [below right of=7]        {$v_{10}$};	
	
	\path
	(1) edge (2)
	(1) edge (3)
	(1) edge (4)
	(2) edge (3)
	(2) edge (4)
	(2) edge (5)
	(3) edge (4)
	(3) edge (5)
	(4) edge (5)
	(5) edge (1)
	(1) edge (6)
	(1) edge (7)
	(1) edge (8)
	(1) edge (9)
	(1) edge (10)
	(2) edge (6)
	(2) edge (7)
	(2) edge (8)
	(2) edge (9)
	(2) edge (10)
	(3) edge (6)
	(3) edge (7)
	(3) edge (8)
	(3) edge (9)
	(3) edge (10)
	(4) edge (6)
	(4) edge (7)
	(4) edge (8)
	(4) edge (9)
	(4) edge (10)
	(5) edge (6)
	(5) edge (7)
	(5) edge (8)
	(5) edge (9)
	(5) edge (10)
	(6) edge (7)
	(6) edge (8)
	(6) edge (9)
	(6) edge (10)
	(7) edge (8)
	(7) edge (9)
	(7) edge (10)
	(8) edge (9)
	(8) edge (10)
	(9) edge (10);

	\end{tikzpicture}
	
	\begin{align*}
	f(X) &= \{0^2, (0+3)^2, (0+3+2)^4, (0+3+2+3)^2\}  
	\\&= \{0^2,3^2,5^4,8^2\}
	\end{align*}
	
	\caption{Mapping a board-pile 10-omino to its corresponding configuration of $K_{10}$.}
	\label{fig:completemappingexample}
\end{figure}

We will now use $f$ to show that the number of period configurations of $K_n$ is equal to the number of board-pile polyominoes containing exactly $n$ unit squares.

\begin{lemma} \label{lem:bpplength}
	Let
	\noindent$X = \{ (d_1, |S_1|), (d_2, |S_2|), (d_3, |S_3|), \dots (d_N, |S_N|)\}$ be a board-pile polyomino with exactly $N$ h-strips. If $1 \leq i \leq N-1$, then $1 \leq d_{i+1} \leq |S_i| + |S_{i+1}| - 1$.
\end{lemma}

\begin{proof}
	Since $S_i$ and $S_{i+1}$ are adjacent h-strips, and polyominoes, by definition, are joined edge on edge, the distance from the least $x$-coordinate of $S_i$ to the greatest $x$-coordinate of $S_{i+1}$ must be less than the sum of the two lengths ($|S_i| + |S_{i+1}|$). So, $d_{i+1}$ must be less than $|S_i| + |S_{i+1}|$. Since $d_{i+1}$ is equal to the difference between the greatest $x$-coordinate in $S_i$ and the least $x$-coordinate in $S_{i-1}$, and since $S_i$ and $S_{i-1}$ are connected edge on edge, $d \geq 1$. Thus, we conclude
	$1 \leq d_{i+1} \leq |S_i| + |S_{i+1}| - 1$.
\end{proof}

\begin{theorem}\label{thm:polyominoisomorphism}
	For all $n \geq 1$, there exists a bijection between the set of all board-pile $n$-ominoes and the set of all period configurations on unlabelled complete graphs.
\end{theorem}

\begin{proof}
	This proof will amount to proving two separate statements: For any board-pile polyomino, $X$, on $n$ unit squares, $n \geq 1$, $f(X)$ is a period configuration of $K_n$, and for any period configuration, $C$, of $K_n$, there is some board-pile polyomino $X$ on $n$ unit squares such that $C = f(X)$.  
	We will first suppose that $X$ is a board-pile $n$-omino and reach that $f(X)$ is a period configuration of the complete graph $K_n$. 
	
	A board-pile polyomino with only a single h-strip maps trivially to a period configuration of $K_n$. In this case, every stack size is equal to 0. Therefore, the configuration is a period configuration with a period length of 1.  
	We will now show that any board-pile polyomino with only two h-strips maps to a period configuration.
	Let $X$ be such a board-pile polyomino. These two h-strips, $S_1$ and $S_2$, map to two sets of vertices, $V_1$ and $V_2$, with distinct stack sizes, $d_1 = 0$ and $d_2$ respectively. We know from Lemma~\ref{lem:bpplength} that in $f(X)$, the vertices of $V_2$ must have between 1 and $|S_1| + |S_2| - 1$ chips. Hence, the vertices of $V_2$ have $d_2$ chips with $1 \leq d_2 \leq |S_1| + |S_2| - 1$. After the initial firing, the vertices of $V_1$ will each have $|S_2|$ chips, having just received from $|S_2|$ richer neighbours, and the vertices of $V_2$ will each have $d_2 - |S_1|$ chips, having just sent to $|S_1|$ poorer neighbours. A result stated in \cite{duffy} and proven in \cite{mullen} states that an integer value can added to or subtracted from each stack size of a configuration without changing the behaviour of the process. So we can normalize the resulting configuration by subtracting $d_2 - |S_1|$ from both totals, leaving $|S_1| + |S_2| - d_2$ chips on each of the vertices in $V_1$ and leaving 0 chips on each of the vertices in $V_2$. Note that since the $x$-distance from the least $x$-coordinate of $S_1$ to the greatest $x$-coordinate of $S_2$ is $d_2$, then the $x$-distance from the least $x$-coordinate of $S_2$ to the greatest $x$-coordinate of $S_1$ must, when added to $d_2$, equal the sum of the two strip lengths. So, the $x$-distance from the least $x$-coordinate of $S_2$ to the greatest $x$-coordinate of $S_1$ is $|S_1| + |S_2| - d_2$. To show that the relative sizes have changed and that in the second firing, the vertices of $V_1$ will send chips to the vertices of $V_2$, we must show that $|S_2| > d_2 - |S_1|$. We know that the maximum value that $d_2$ can take on is $|S_1| + |S_2| - 1$.
	So, 
	
	\begin{align*}
	|S_1| + |S_2| - 1 &\geq d_2\\
	|S_2| &\geq d_2 - |S_1| + 1\\
	|S_2| &> d_2 - |S_1|  
	\end{align*}
	
	Thus, we can conclude that the vertices of $V_1$ are now richer than the vertices of $V_2$.

	This gives us that the configuration following the initial firing, call it $[f(X)]'$, is itself equal to $f(X')$ for some board-pile n-omino $X'$. In fact, $X'$ is the board-pile $n$-omino created by reflecting $X$ about the horizontal axis since $[f(X)]'$ represents an interchange of the relative stack sizes of the two sets of vertices, $V_1$ and $V_2$, corresponding to the two h-strips in $X$ (See Figure~\ref{fig:2stripexample}). So, after the second firing, the vertices of $V_2$ will have $|S_1|$ chips and the vertices of $V_1$ will have $|S_1| - d_2$ chips. By adding $d_2 - |S_1|$ to both totals (to counteract our subtracting of $d_2 - |S_1|$ chips previously), we get back where we started with the vertices of $V_1$ having 0 chips and the vertices of $V_2$ having $d_2$ chips. So, we have that $f(X)$ is a period configuration.

	\begin{figure}[H]
		\centering
		\begin{tikzpicture}[line width = 2pt,scale = 1.0]
		\path [draw, help lines, opacity=.5]  (0,0) grid (5,5);
		\foreach \i in {1,...,5} \draw (\i,2.5pt) -- +(0,-5pt) node [anchor=north, font=\small] {$\i$} (2.5pt,\i) -- +(-5pt,0) node [anchor=east, font=\small] {$\i$};
		\draw (1,0) -- (1,1);
		\draw (1,1) -- (2,1);
		\draw (0,1) -- (1,1);
		\draw (1,0) -- (2,0);
		\draw (2,0) -- (2,1);
		\draw (2,1) -- (3,1);
		\draw (2,2) -- (2,1);
		\draw (1,2) -- (2,2);
		\draw (3,1) -- (3,2);
		\draw (2,2) -- (3,2);
		\draw (0,1) -- (0,2);
		\draw (1,1) -- (1,2);
		\draw (0,2) -- (1,2);
		\draw (2,0) -- (3,0);
		\draw (3,0) -- (3,1);
		\draw (3,1) -- (4,1);
		\draw (4,1) -- (5,1);
		\draw (4,1) -- (4,2);
		\draw (5,1) -- (5,2);
		\draw (3,2) -- (4,2);
		\draw (4,2) -- (5,2);
		\node[draw=none] at (-0.5,0.5) {$S_1$};
		\node[draw=none] at (-0.5,1.5) {$S_2$};
		\node[draw=none] at (1.5,0.5) {$0$};
		\node[draw=none] at (2.5,0.5) {$0$};
		\node[draw=none] at (0.5,1.5) {$4$};
		\node[draw=none] at (1.5,1.5) {$4$};
		\node[draw=none] at (2.5,1.5) {$4$};
		\node[draw=none] at (3.5,1.5) {$4$};
		\node[draw=none] at (4.5,1.5) {$4$};
		\end{tikzpicture}
		
		\begin{tikzpicture}[line width = 2pt,scale = 1.0]
		\path [draw, help lines, opacity=.5]  (0,0) grid (5,5);
		\foreach \i in {1,...,5} \draw (\i,2.5pt) -- +(0,-5pt) node [anchor=north, font=\small] {$\i$} (2.5pt,\i) -- +(-5pt,0) node [anchor=east, font=\small] {$\i$};
		\draw (1,0) -- (1,1);
		\draw (1,1) -- (2,1);
		\draw (0,1) -- (1,1);
		\draw (1,0) -- (2,0);
		\draw (2,0) -- (2,1);
		\draw (2,1) -- (3,1);
		\draw (2,2) -- (2,1);
		\draw (1,2) -- (2,2);
		\draw (3,1) -- (3,2);
		\draw (2,2) -- (3,2);
		\draw (5,0) -- (5,1);
		\draw (1,1) -- (1,2);
		\draw (4,1) -- (5,1);
		\draw (2,0) -- (3,0);
		\draw (3,0) -- (3,1);
		\draw (3,1) -- (4,1);
		\draw (4,0) -- (4,1);
		\draw (4,0) -- (5,0);
		\draw (3,0) -- (4,0);
		\draw (0,0) -- (1,0);
		\draw (0,0) -- (0,1);
		\node[draw=none] at (-0.5,0.5) {$S_2$};
		\node[draw=none] at (-0.5,1.5) {$S_1$};
		\node[draw=none] at (1.5,1.5) {$5$};
		\node[draw=none] at (2.5,1.5) {$5$};
		\node[draw=none] at (0.5,0.5) {$2$};
		\node[draw=none] at (1.5,0.5) {$2$};
		\node[draw=none] at (2.5,0.5) {$2$};
		\node[draw=none] at (3.5,0.5) {$2$};
		\node[draw=none] at (4.5,0.5) {$2$};
		\end{tikzpicture}
		
		\caption{Board-pile with two strips ``flipping" upon the firing of the related complete graph configuration.}	
		\label{fig:2stripexample}
	\end{figure}
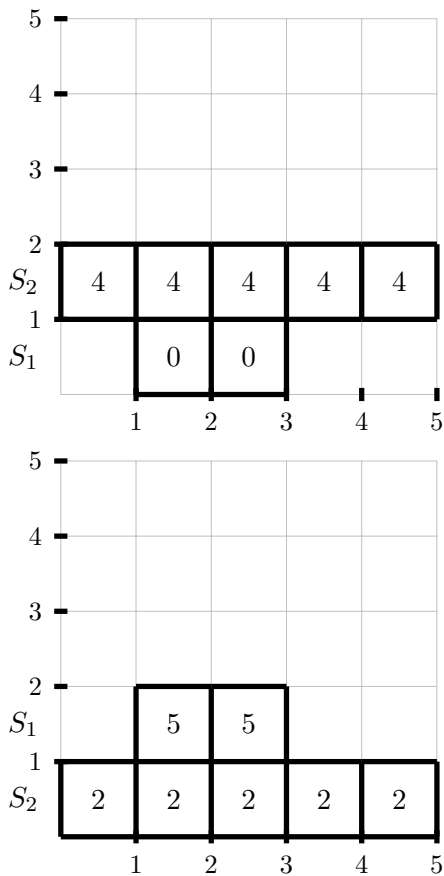
	
	\FloatBarrier

	We will use this case as the basis of an induction. We will induct on the number of h-strips to show that for all board-pile $n$-ominoes, $X$, the initial firing of $f(X)$ yields $f(X')$ (up to an addition of a constant to each of the stack sizes) where $X'$ is the board-pile n-omino created by reflecting $X$ about the horizontal axis. This will imply that $f(X)$ is a period configuration of some complete graph because two firings will return the original configuration.
	
	Our inductive hypothesis is that for all board-pile polyominoes $X$ with at most $k$ h-strips, the initial firing of $f(X)$ yields $f(X')$ (up to an addition of a constant to each of the stack sizes) where $X'$ is the board-pile polyomino created by reflecting $X$ about the horizontal axis. 
	
	Now suppose we have a board-pile polyomino, $Y$, with exactly $k+1$ h-strips. Then, $f(Y)$ is some configuration of some complete graph. Thus, by Lemma~\ref{lem:bpplength}, $d_{k+1}$ is bounded so that $1 \leq d_{k+1} \leq |S_{k}| + |S_{k+1}| - 1$. We know by our inductive hypothesis that if $S_{k+1}$ were removed from $Y$, that the resulting polyomino would map to a period configuration of some complete graph. Also, the vertices of $V_{k+1}$ in $f(Y)$ will act as source vertices enriching every other vertex upon firing. The addition of such a source strip cannot affect the relative stack sizes of the other vertices within the graph.
	We know from our base case that the board-pile polyomino $Z$ composed of just the squares of $S_k$ and $S_{k+1}$ is such that $f(Z') = [f(Z)]'$. Therefore, since the effect of the vertices in $V(f(Y)) \setminus V(f(Z))$ will enrich the remaining stack sizes of the graph equally upon firing, we have that $Y$ satisfies our criteria with $f(Y')=[f(Y)]'$.

	We now present the second portion of the proof where we begin by supposing that $C$ is a period configuration of $K_n$ and reach that there exists some board-pile $n$-omino $X$ such that $C = f(X)$.

	Suppose the vertices of $C$ have $N$ distinct stack sizes
	
	\begin{align*}
	\{0^{a_1},(\sum_{i=1}^{2} d_{i})^{a_2},(\sum_{i=1}^{3} d_{i})^{a_3}, \dots (\sum_{i=1}^{N} d_{i})^{a_N}\}
	\end{align*}

	\noindent in ascending order for some set $a_1, a_2, \dots, a_N \in \mathbb{N}$ and some set $d_1=0, d_2, d_3, \dots, d_N$ with $d_2, \dots, d_N \in \mathbb{N}.$ Let $V_j$ be the set of all vertices in $C$ with stack size $(\sum_{i=1}^{j} d_{i})^{a_j}$ for all $j \leq N$. Let $X$ be a collection of $N$ h-strips on a plane with exactly one h-strip per row for $y=1,2, \dots, N$. Call these h-strips $S_1, S_2, \dots, S_N$, containing $|V_1|, |V_2|, \dots, |V_N|$ unit squares, respectively, with $S_i$ spanning y-coordinates $i-1$ to $i$ for all $i \leq N$. Arrange these strips so that $d_j$ is equal to the $x$-distance from the leftmost coordinate of $S_{j-1}$ to the rightmost coordinate of $S_j$.
	$X$ is a board-pile polyomino if and only if it is a connected plane figure. 
	
	In particular, we must prove that $d_{j+1} \leq |V_j| + |V_{j+1}| - 1$ for all $j$. This will imply that every pair of strips $|S_j|$ and $|S_{j+1}|$ are connected edge on edge, proving that a single board-pile polyomino will result rather than a number of disconnected board-pile polyominoes in the plane.

	By contradiction, suppose $d_{j+1} > |V_j| + |V_{j+1}| - 1$ for some $j$. Then, following the initial firing, the vertices of $V_j$ would each have $(\sum_{i=1}^{j} d_{i}) + |V_{j+1}| + r$ chips, where $r$ represents the difference between the number of vertices richer than those in $V_{j+1}$ and the number of vertices poorer than those in $V_j$. Also, following the initial firing, the vertices of $V_{j+1}$ would each have $(\sum_{i=1}^{j+1} d_{i}) - |V_j| + r$ chips. Since $C$ is a period configuration, we know that $(\sum_{i=1}^{j} d_{i}) + |V_{j+1}| + r > (\sum_{i=1}^{j+1} d_{i}) - |V_i| + r$. But, this implies that $d_{j+1} < |V_{j+1}| + |V_{j}|$ which implies that $d_{j+1} \leq |V_{j+1}| + |V_j| - 1$ 
	which contradicts our assumption. So, for all period configurations, $C$, $C = f(X)$, for some board-pile polyomino $X$.
\end{proof}

Since we have shown $f$ to be bijective, we can now count the number of period configurations of unlabelled complete graphs with a given number of vertices by using previous results regarding board-pile polyominoes.

\begin{corollary}
	The number of period configurations of a complete graph on $n$ vertices follows the recurrence relation $a_n = 5a_{n-1} - 7a_{n-2} + 4a_{n-3}$ for $n \geq 5$ with initial values $a_1 = 1$, $a_2 = 2$, $a_3 = 6$, and $a_4 = 19$.
\end{corollary}

\begin{proof}
The number of board-pile polyominoes of size $n$ can be found in the OEIS, sequence A001169 \cite{oeis}, and the values are exhibited in Table~\ref{tab:board-pilepolyominoes}. The generating function is $\frac{x(1-x)^3}{(1 - 5x + 7x^2 - 4x^3)}$ \cite{oeis}. This sequence of numbers follows the recurrence relation $a_n = 5a_{n-1} - 7a_{n-2} + 4a_{n-3}$ for $n \geq 5$.\cite{klarner}

\begin{table}[H] 
	\centering
	\begin{tabular}{|c|c|}
		\hline
		$n$  & $\#$ of board-pile polyominoes\\
		\hline
		$1$  & 1\\
		\hline
		$2$  & 2\\
		\hline
		$3$  & 6\\
		\hline
		$4$  & 19\\
		\hline
		$5$  & 61\\
		\hline
		$6$  & 196\\
		\hline
		$7$  & 629\\
		\hline
		$8$  & 2017\\
		\hline
		$9$  & 6466\\
		\hline
		$10$  & 20727\\
		\hline
		$11$  & 66441\\
		\hline	
	\end{tabular}
	\caption{Number of board-pile polyominoes containing $n$ unit squares for $1 \leq n \leq 11$.}
	\label{tab:board-pilepolyominoes}
\end{table}

By Theorem~\ref{thm:polyominoisomorphism}, the number of period configurations of an unlabelled complete graph up to equivalence is equal to the number of board-pile polyominoes of size $n$. Thus, the number of period configurations of a complete graph on $n$ vertices follows the recurrence relation $a_n = 5a_{n-1} - 7a_{n-2} + 4a_{n-3}$ for $n \geq 5$ with initial values $a_1 = 1$, $a_2 = 2$, $a_3 = 6$, and $a_4 = 19$. 

\end{proof}

In order to find the explicit formula, we must perform some algebra:

\begin{align*}
a_n &= 5a_{n-1} - 7a_{n-2} + 4a_{n-3}\\
a_n - 5a_{n-1} + 7a_{n-2} - 4a_{n-3} &= 0 && \text{Let $a_n = x^n$}\\
x^n - 5x^{n-1} + 7x^{n-2} - 4x^{n-3} &= 0\\
x^{n-3}(x^3 - 5x^{2} + 7x - 4) &= 0\\
x^3 - 5x^2 + 7x - 4 &= 0
\end{align*}

The roots of this equation are $\alpha_1 \approx 3.2056$, $\alpha_2 \approx 0.8972 - 0.6655i$, and $\alpha_3 \approx 0.8972 + 0.6655i$.

The solution for the $k^{th}$ value of this recurrence is

\vspace{1cm}

$$
\sum _{i=1}^{3} - {\frac { \left( -7 {{\alpha_i}}^{-2} + 13 {\alpha_i}^{-1} - 5  \right) 
		\left( {{\alpha_i}} \right) ^{k}}{ \left( 192\,{{\alpha_i}}^{-2} - 224\,{
			{\alpha_i}^{-1}} + 80\, \right) {\alpha_i}^{-1}}}
$$

\vspace{1cm}

This can be rewritten as $a_k = c_1(\alpha_1)^{\;k} + c_2(\alpha_2)^{\;k} + c_3(\alpha_3)^{\;k}$. The dominating term, out of these three roots, is the one which has the greatest modulus. These values are roughly 3.2056, 1.1171, and 1.1171. Thus, in the equation $a_k = c_1(\alpha_1)^{\;k} + c_2(\alpha_2)^{\;k} +c_3(\alpha_3)^{\;k} +c_4(\alpha_4)^{\;k}$, the dominant term is $c_1(\alpha_1 )^{\;k}\approx (0.1809)(3.2056)^{\;k}$.

\begin{corollary}
	$a_k$ has an asymptotic value of $0.1809 \times 3.2056^k$.
\end{corollary}

\section{Period Configurations of Labelled Complete Graphs}

How many period configurations exist on \textit{labelled} complete graphs? We know that a configuration $C$ is a period configuration if and only if $f^{-1}(C)$ is a board-pile polyomino. 

Our method must involve finding all labelled ordered partitions of the vertices. Then, given a labelled ordered partition, we must determine the number of ways that the $i^{th}$ part can be connected edge on edge with parts $i-1$ and $i+1$. 

Let $K_n$ be a complete graph. The number of labelled ordered partitions of a complete graph is just the ordered Bell numbers \cite{mullen}. We can view this as all of the ways that a set of labelled unit squares can be grouped together into ordered strips but without yet affixing the strips together to form a polyomino. 

For every possible ordered partition, we must determine all of the possible stack sizes which could be shared by the vertices in each respective part. Recall that this is equivalent to determining all of the ways that an ordered set of strips can be oriented with respect to each other so as to create a polyomino. 
 
For each period orientation of $K_n$, assign all possible stack sizes by using the product $\prod_{i=1}^{N}(|S_{i-1}+ |S_i| - 1)$ (the product of the number of ways each strip $S_i$ with length $|S_i|$ can be adjacent to its neighbours). Give the period orientations of $K_n$ some ordering from $1$ to $R(K_n)$. Let $|S^j_{i}|$ be the length of the $i^{th}$ strip in the $j^{th}$ period orientation of $K_n$. Then we reach that the number of labelled period configurations that exist on $K_n$ is 

$$\sum_{j=1}^{R(K_n)} \prod_{i=1}^{N}(|S^j_{i-1}| + |S^j_i| - 1).$$

\begin{theorem}
	The number of period configurations which exist on a labelled complete graph with $n$ vertices is
	
	$$\sum_{j=1}^{R(K_n)} \prod_{i=1}^{N}(|S^j_{i-1}| + |S^j_i| - 1).$$
\end{theorem}

\section{Open Problems}

Prior to this result using polyominoes, the only results for counting period configurations in Parallel Diffusion were on paths and stars \cite{mullen}. It remains an open problem to count all of the period configurations up to equivalence on many other families of graphs like cycles, trees, and complete bipartite graphs. Much of the work on Parallel Diffusion since Long and Narayanan \cite{long} proved the conjecture from Duffy et al. \cite{duffy} has been on variations to the rules. However, there remain many questions about pre-period configurations that have not been answered in addition to the questions about counting period configurations. As for the relationship between Parallel Diffusion and polyominoes, very little is known. Are there properties of polyominoes in general (rather than just board-pile polyominoes) that can aid in our understanding of complete graphs in Parallel Diffusion? Conversely, can the research into Parallel Diffusion answer any questions in the field of polyominoes? Since there seems to be no reason to believe that polyominoes will be useful in analyzing any graphs other than the complete graph in Parallel Diffusion, the result that would be most satisfying would be one that relates general polyominoes to pre-period configurations of the complete graph and characterizes the period configurations (board-pile polyominoes) based on these pre-period configurations.

\end{document}